\documentclass[12pt]{article}
\usepackage{a4}
\usepackage{amsthm}
\usepackage{amsfonts}
\usepackage{amssymb}
\usepackage{amsmath}
\usepackage{cite}
\usepackage{xcolor}
\usepackage{comment}
\usepackage{epsfig}
\usepackage{verbatim}

\usepackage{tikz}
\tikzstyle{aNode} = [circle, fill = black]
\tikzstyle{bNode} = [circle,draw = black, thick]

\newcommand{\ppoints}[1]{%
\begin{tikzpicture}[inner sep = 0.7pt, #1]%
\node (1) at (0,-2) [aNode]{};
\node (3) at (1.5,-2) [aNode]{};
\node (2) at (0.75,-1) [aNode]{};
\end{tikzpicture}%
}

\def\points{\ppoints{scale=0.08}}

\newtheorem{theorem}{Theorem}

\newtheorem{lemma}[theorem]{Lemma}

\newtheorem{definition}[theorem]{Definition}
\newtheorem{conjecture}[theorem]{Conjecture}

\newcommand\NN{{\mathbb N}}
\newcommand\ZZ{{\mathbb Z}}

\newcommand\cross{\operatorname{Cross}}
\newcommand\fP{\mathfrak{P}}
\newcommand\cP{\mathcal{P}}
\newcommand\ep{\varepsilon}

\newcommand\pu{{\pi_{\points}}}

\DeclareTextCompositeCommand{\v}{OT1}{l}{l\nobreak\hspace{-.1em}'}
\DeclareTextCompositeCommand{\v}{OT1}{t}{t\nobreak\hspace{-.1em}'\nobreak\hspace{-.15em}}

\title{Uniform Tur\'an density beyond $3$-graphs}
\author{Ander Lamaison\thanks{Extremal Combinatorics and Probability Group (ECOPRO),  Institute for Basic Science (IBS), Daejeon, South Korea. Supported by IBS-R029-C4. \textbf{Email address}: ander@ibs.re.kr}}
\date{}

\begin{document}

\maketitle

\begin{abstract}In the 1980s, Erd\H{o}s and S\'os first introduced an extremal problem on hypergraphs with density constraints. Given an $r$-uniform hypergraph $F$ (or $r$-graph for short), its uniform Tur\'an density $\pu(F)$ is the smallest value of $d$ in which every hypergraph $H$ in which every linear-sized subhypergraph of $H$ has edge density at least $d$ contains $F$ as a subgraph. The first non-zero value of $\pu(F)$ was not found until 30 years later.

Progress in studying the set of values of the uniform Tur\'an density of $r$-graphs has been uneven in terms of $r$: to this day there are infinitely many non-zero values known for $r=3$, a single non-zero value known for $r=4$ and none for $r\geq 5$. In this paper we obtain the first explicit values of $\pu$ for all uniformities, by proving that for every $r\geq 3$ there exist $r$-graphs $F$ with $\pu(F)=1/4$ and with $\pu(F)=\binom{r}{2}^{-\binom{r}{2}}$.\end{abstract}

\section{Introduction}

One of the central types of problem in extremal graph and hypergraph theory is the so-called Tur\'an-type problems. These problems concern the largest number of edges in a hypergraph with a prescribed number of vertices and avoiding a fixed subgraph.

One of the most basic Tur\'an-type problems is determining the Tur\'an density of hypergraphs. Let $F$ be an $r$-uniform hypergraph (or $r$-graph for short). Its Tur\'an density, denoted by $\pi(F)$, is the largest real number $d$ such that, for all $\varepsilon>0$ and for all $N$, there exists an $r$-graph $H$ on $n\geq N$ vertices with edge density at least $d-\varepsilon$ (i.e. has at least $(d-\varepsilon)\binom{n}{r}$ edges) and which is $F$-free (i.e. does not contain $F$ as a subgraph).

Erd\H{o}s and Stone~\cite{ErdS46} showed that the Tur\'an density of graphs is determined by the chromatic number $\chi(G)$, specifically $\pi(G)=\frac{\chi(G)-2}{\chi(G)-1}$. No similar characterization is known or conjectured in the case of $r$-graphs with $r\geq 3$. Erd\H{o}s offered \$500 for determining the Tur\'an density of any complete hypergraph $K_n^{(r)}$ with $n>r\geq 3$, and \$1000 for determining all values of $\pi(K_n^{(r)})$. No such value is known as of today, and even the case $F=K_4^{(3)-}$, the hypergraph obtained from $K_4^{(3)}$ by removing one edge, remains open.

For many choices of $F$, the densest $F$-free hypergraphs have edges which are very unevenly distributed within its vertex set. For example, if $F=K_4^{(3)}$, the conjectured value of $\pi(F)$ is $5/9$, and Tur\'an~\cite{Tur61} constructed $F$-free $3$-graphs $H_n$ on $n$ vertices with edge density tending to $5/9$. In $H_n$, the vertex set can be split into three independent sets of size $n/3$. In other words, despite the \emph{global} edge density being $5/9$, the \emph{local} edge-density in some large subsets of the vertex set is as low as $0$.

We will now formalize this notion of local density. Let $H$ be an $r$-graph. We say that $H$ is $(d,\ep,\points)$-dense if every set $S\subseteq V(H)$ with $|S|\geq \varepsilon |V(H)|$ satisfies $e(S)\geq (d-\varepsilon)\binom{|S|}{r}$. We say that a sequence of hypergraphs $\{H_n\}_{n=1}^\infty$ is \emph{locally $d$-dense} if there exists a sequence $\{\varepsilon_n\}_{n=1}^\infty$ such that each $H_n$ is $(d,\varepsilon_n,\points)$-dense, $\varepsilon_n\rightarrow 0$, and $|V(H_n)|\rightarrow\infty$.

\begin{definition}
    The uniform Tur\'an density of an $r$-graph $F$, denoted by $\pu(F)$, is the maximum value of $d$ such that there exists a locally $d$-dense sequence of $F$-free $r$-graphs.
\end{definition}

Erd\H{o}s and S\'os\cite{ErdS82} introduced uniform Tur\'an density in the 1980s, but it would take over thirty years for the first non-zero value to be determined, when Glebov, Kr\'al' and Volec~\cite{GleKV16} proved that $\pu(K_4^{(3)-})=1/4$.

The uniform Tur\'an density of every graph is 0: for every $\varepsilon>0$, every locally $\varepsilon$-dense sequence of graphs contains arbitrarily large cliques. Therefore, the smallest value of $r$ for which the uniform Tur\'an density of $r$-graphs is a non-trivial parameter is 3. There are several families of $3$-graphs whose uniform Tur\'an density is known, such as tight cycles~\cite{BucCKMM23} and large stars~\cite{LamW24}.

In addition to individual values of $\pu(F)$, other results have been proved about the uniform Tur\'an density of $3$-graphs. Reiher, R\"odl and Schacht~\cite{ReiRS18} characterized the $3$-graphs $F$ with $\pu(F)=0$, and proved that the uniform Tur\'an density of a $3$-graph cannot lie in the interval $(0,1/27)$. Garbe, Kr\'al' and the author~\cite{GarKL24} constructed $3$-graphs $F$ with $\pu(F)=1/27$, meaning that $1/27$ is in fact the minimum positive value of $\pu$ in $3$-graphs.

One reason why so many results have been proved about the uniform Tur\'an density of $3$-graphs is the development of certain methods for its study. The first systematic method was described by Reiher~\cite{Rei20}, who introduced a structure called reduced hypergraphs, and showed its connection to uniform Tur\'an density. Another useful tool is palettes, a way of generating locally dense sequences of $3$-graphs introduced by Reiher based on constructions by Erd\H{o}s and Hajnal~\cite{ErdH72} and R\"odl~\cite{Rod86}. Recently, the author~\cite{Lam24} proved that, for any $3$-graph $F$, $\pu(F)$ is the supremum of the densities of palette constructions not containing $F$.

In contrast, very little is known about the uniform Tur\'an density of $r$-graphs with $r\geq 4$. Kohayakawa, Nagle, R\"{o}dl and Schacht~\cite{KohNRS10} proved that if $F$ is linear (i.e. no two edges intersect in more than one vertex) then $\pu(F)=0$. There is only one non-zero value of $\pu(F)$ known, which was found as a byproduct of its classical Tur\'an density. Let $F$ be the unique $4$-graph on five vertices with three edges. Gunderson and Semeraro~\cite{GunS17} proved that $\pi(F)=1/4$. The sequence of $F$-free hypergraphs with edge density $1/4$ that they generated is also locally $1/4$-dense, so as a corollary, one can deduce that $\pu(F)=1/4$.

Among the properties of the uniform Tur\'an density of $r$-graphs that remain unknown for $r\geq 4$, one of them is the characterization of the hypergraphs $F$ satisfying $\pu(F)=0$. We will propose one such characterization in Conjecture~\ref{conj:dens0}. 

The structure yielded by the hypergraph regularity lemma, which lies at the heart of Reiher's reduced hypergraphs, is one of the causes behind this difficulty increase when $r\geq 4$. The regularity lemma produces a \emph{layered partition}, in which we obtain a partition of each of the sets $\binom{V}{1}, \binom{V}{2}, \dots, \binom{V}{r-1}$, and thanks to a tehnique used by Reiher, the first layer can be `tamed' in locally dense sequences of hypergraphs. This means that for $r=3$, one only needs to pay attention to one layer of the regularity partition, whereas for $r\geq 4$ there are multiple layers with a non-trivial behavior.

There are other exact results known using more restrictive notions of local density, introduced by Reiher, R\"odl and Schacht~\cite{ReiRS18c} and whose corresponding Tur\'an densities are denoted by $\pi_i(F)$. In this hierarchy, $\pi_0(F)$ and $\pi_1(F)$ correspond to the classical and uniform Tur\'an densities, respectively. Lin, Wang and Zhou~\cite{LinWZ23} showed that the minimum positive value of $\pi_{r-2}(F)$ among $r$-graphs $F$ is $r^{-r}$.

In this paper we find the first non-zero values of the uniform Tur\'an density of $r$-graphs for general $r$. First we will prove that $1/4$ is a value for every $r\geq 3$, generalizing the results from Glebov, Kr\'al' and Volec for $r=3$, and by Gunderson and Semeraro for $r=4$.

\begin{theorem}\label{thm:main14}
    For every $r\geq 3$ there exists an $r$-graph $F$ with $\pu(F)=1/4$.
\end{theorem}

Let $\pi_r=\binom{r}{2}^{-\binom{r}{2}}$. We will also prove the existence of $r$-graphs with uniform Tur\'an density equal to $\pi_r$, generalizing the case $r=3$ by Garbe, Kr\'al' and the author. This number is significant because we conjecture that this is the minimum positive value of the uniform Tur\'an density of $r$-graphs. In other words, we expect the set of values of the uniform Tur\'an density to present a `jump': for every $r$-graph $F$ we either have $\pu(F)=0$ or $\pu(F)\geq\pi_r$. The case $r=3$ was proved by Reiher, R\"odl and Schacht~\cite{ReiRS18}. This will be discussed in the concluding remarks.

\begin{theorem}\label{thm:mainpi}
    For every $r\geq 3$ there exists an $r$-graph $F$ with $\pu(F)=\binom{r}{2}^{-\binom{r}{2}}$.
\end{theorem}

The organization of this paper is as follows. In Section~\ref{sec:prelim} we will list some results that will be used in the proof of our main theorems. Section~\ref{sec:descript} will introduce the concept of descriptive sequences, and will split Theorem~\ref{thm:main14} and Theorem~\ref{thm:mainpi} each into an `existence' statement and a `density' statement. The existence statements will be proved in Section~\ref{sec:exist}. Two important tools in the proof of the density statements, namely reduced hypergraphs and hypergraph regularity, will be introduced in Section~\ref{sec:reduced} and Section~\ref{sec:reg}, respectively. Section~\ref{sec:density} will be devoted to proving the density statements. Finally, in Section~\ref{sec:concluding} we will present some concluding remarks.

\section{Preliminaries}\label{sec:prelim}

One tool that we will use is Ramsey's theorem, which allows us to find monochromatic cliques in edge colorings of complete hypergraphs.

\begin{theorem}[Ramsey's theorem~\cite{Ram30}]
    Let $r$, $m$ and $k$ be positive integers. There exists a number $n$ such that, whenever the edges of a complete $r$-uniform hypergraph on $n$ vertices are colored using $k$ colors, there exists a subset of $m$ vertices in which every edge receives the same color.
\end{theorem}

We will denote the least value of $n$ with this property as $R_r(m,k)$. These numbers are often referred to as Ramsey numbers.

\begin{definition}Let $S\subseteq\ZZ^d$ be a set, and let $\prec$ be a linear order on the set $S$. We say that $S$ is \emph{sorted by the $i$-th coordinate} if for any $\vec x=(x_1,\dots, x_d),\vec y=(y_1, \dots, y_d)\in S$, if $x_i<y_i$ holds then $\vec x \prec \vec y$. We say that $S$ is \emph{inversely sorted by the $i$-th coordinate} if for any $\vec x,\vec y\in S$, if $x_i<y_i$ holds then $x\succ y$. 

We say that $S$ is \emph{sorted by layers} if there exists some $i$ such that $S$ is sorted by the $i$-th coordinate or inversely sorted by the $i$-th coordinate. \end{definition}

The following lemma follows directly from a theorem of Fishburn and Graham~\cite{FisG93}, by noting that the \emph{lexicographic order} defined in that paper is in particular sorted by layers.

\begin{lemma}\label{lem:FisGra} For every $k,d\in\NN$ there exists $N=N(k,d)$ with the following property: for every linear ordering $\prec$ on the set $[N]^d$, there exist sets $S_1, S_2, \dots, S_d$, each with size $k$, such that $S_1\times S_2\times \dots\times S_d$ is sorted by layers by $\prec$.\end{lemma}

Observe that when $d=1$, setting $N=(k-1)^2+1$, we obtain the Erd\H{o}s-Szekeres theorem: any ordering of $[N]$ contains a monotone subsequence with $k$ terms. 

The next lemma allows us to find subsets of ordered sets which span disjoint intervals. The method used in this proof has been implicitly used before (see for example~\cite{BarV98}, which is a result in the same spirit), but perhaps the most explicit description of the algorithm can be found in the solution of Problem 5 from the International Mathematical Olympiad 2017. 

\begin{lemma}\label{lem:IMO} Let $k$ be a positive integer. Let $A_1, A_2, \dots, A_k$ be disjoint finite sets, and let $A$ be their union. For every linear order $\prec$ on $A$, there exist sets $B_1, \dots, B_k$ with $B_i\subseteq A_i$ and $|B_i|\geq |A_i|/k$, satisfying the following property: for any $i\neq j$ there do not exist elements $x,z\in B_i$ and $y\in B_j$ such that $x\prec y\prec z$. \end{lemma}

\begin{proof} We proceed by induction on $k$. The case $k=1$ is trivial. Now assume that the statement holds for $k-1$. We can assume that no set $A_i$ is empty, otherwise apply induction to the remaining sets. Consider a set $C$, which is initially empty. Add elements to $C$ one by one, from smallest to largest, until the first time some set $A_i$ satisfies $|A_i\cap C|\geq |A_i|/k$. W.l.o.g. we can assume $i=k$. We set $B_k=A_k\cap C$, and $A'_i=A_i\setminus C$ for all $1\leq i \leq k-1$. 

By construction there is no element of $B_k$ between two elements of $A'_i$, and vice versa. Also, since the sets $A_i$ are all disjoint, by the minimality of $C$ we have that $|A_i\cap C|<|A_i|/k$, and thus $|A'_i|> \frac{k-1}{k}|A_i|$. Applying the induction hypothesis to $A'_1, \dots, A'_{k-1}$ we obtain the desired sets $B_1,\dots, B_{k-1}$, which satisfy $|B_i|\geq |A'_i|/(k-1)\geq|A_i|/k$.\end{proof}

Finally, we will use a family of linear hypergraphs that satisfies a certain quasirandomness condition. Remember that a hypergraph is said to be \emph{linear} if every pair of edges intersects in at most one vertex.

\begin{definition} A sequence of $r$-graphs $ \{H_n\}_{n=1}^\infty$ is said to be \emph{nowhere-empty} if for every $c>0$ there exists $n_0$ with the following property: for all $n\geq n_0$, for all disjoint subsets $S_1, S_2, \dots, S_r$ of vertices, each with size at least $c|V(H_n)|$, there exists some edge of $H_n$ that intersects all sets $S_i$.\end{definition}

\begin{lemma}\label{lem:nowhere} Let $r\geq 3$. There exists a nowhere-empty sequence of linear $r$-graphs $\{H_n\}_{n=1}^\infty$, where each $H_n$ has $n$ vertices. \end{lemma}

\begin{proof} We will start by considering the Erd\H{o}s-R\'enyi random hypergraph $G_n=G^{(r)}(n,p)$, with $p=n^{-r+3/2}$. This is the random hypergraph obtained from a set of $n$ vertices, where for every $r$-tuple of vertices we add an edge with probability $p$, independently of the other choices. In the following, by ``with high probability'' we mean with probability tending to 1 as $n$ tends to infinity.

Let $T$ be the set of all edges of $G_n$ that intersect another edge in at least two vertices. It is clear that $G_n-T$ is a linear hypergraph. We claim that, with high probability, $|T|\leq n^{5/4}$. Indeed, let us count the number of pairs of edges intersecting in at least two vertices. Every such pair can be written in the form $(v_1, v_2, w_1, w_2, \dots, w_{r-2}, w'_1, w'_2, \dots, w'_{r-2})$, where the two edges are $v_1v_2w_1w_2\dots w_{r-2}$ and $v_1v_2w'_1w'_2\dots w'_{r-2}$. Such an expression is generally not unique: the order of the $v_i, w_i$ and $w'_i$ might be altered, and if the two edges intersect on more than two vertices then the choice of $v_1$ and $v_2$ is also not unique. What matters is that every pair of edges with intersection at least two admits such an expression. 

Given a $2r-2$-tuple of vertices, the probability that it corresponds in this way to a pair of edges of $G_n$ is at most $p^2$. Therefore, by linearity of expectation, the expected number of pairs of intersecting edges is at most $n^{2r-2}p^2=n$. By Markov's inequality, the probability that there are at least $n^{5/4}/2$ pairs of intersecting edges is at most $2n^{-1/4}$. Thus, with high probability $|T|\leq n^{5/4}$.

Next we will show that, with high probability, for every choice of disjoint $S_1, S_2, \dots, S_r$, each with size at least $n^{1-\frac{1}{6r}}$, there is an edge of $G_n-T$ that intersects each $S_i$. This is enough to prove our statement by taking $H_n=G_n-T$. Observe that there are at most $(r+1)^n$ choices of the subsets $S_i$: for each vertex, we can decide whether to include it in one of the sets $S_i$, and if so in which one, yielding $r+1$ options for each vertex. Therefore, it is enough to show that for each fixed choice of sets $S_i$, conditioned on $|T|\leq n^{5/4}$, the probability that there does not exist an edge of $G_n-T$ intersecting every $S_i$ is $o((r+1)^{-n})$. 

Indeed, the fact that no edge of $G_n-T$ intersects all $S_i$ implies that at most $|T|$ edges of $G_n$ intersect every $S_i$ (and all of them belong to $T$). In particular, if $|T|\leq n^{5/4}$, then no more than $n^{5/4}$ edges of $G_n$ intersect every $S_i$. Observe that the number of edges of $G_n$ that intersect every $S_i$ behaves like a binomial distribution: for each of the $|S_1|\cdot |S_2|\cdot\dots\cdot |S_r|$ choices of one vertex from each subset, whether or not this $r$-tuple forms an edge of $G_n$ is independent and happens with probability $p$. 

The probability that fewer than $n^{5/4}$ of these $r$-tuples become an edge of $G_n$ is at most \[\binom{(n^{1-r/6})^r}{n^{5/4}}(1-p)^{(n^{1-r/6})^r-n^{5/4}}=n^{O(n^{5/4})}e^{-\Omega(pn^{r-1/6})}=e^{-\Omega(n^{4/3})}=o\left((r+1)^{-n}\right).\]

With high probability, $|T|\leq n^{5/4}$ and all choices of $S_1, S_2, \dots, S_r$ span more than $n^{5/4}$ edges of $G_n$, so it spans at least one edge of $H_n$. This concludes our proof.\end{proof}

\section{Consistent orderings and descriptive sequences}\label{sec:descript}

We will now present some notions relating to hypergraphs, which will be necessary to describe the family of hypergraphs for which we will give the exact value of the uniform Tur\'an density.

Kohayakawa, Nagle, R\"odl and Schacht~\cite{KohNRS10} showed that the uniform Tur\'an density of every linear hypergraph is 0. Therefore, in order to keep edge intersections as simple as possible, we will define the following class of hypergraphs. We say that a hypergraph $H$ is \emph{quasi-linear} if for every edge $e\in E(H)$ there exists a unique other edge $f$ such that $|e\cap f|=2$, and all $f'\in E(H)\setminus\{e,f\}$ satisfy $|e\cap f'|\leq 1$. In other words, each edge has a unique \emph{twin}, twin pairs of edges have intersection size two, and non-twin pairs of edges have intersection size at most one.

We will consider orderings of the vertex set of $H$. We will be interested in the relative order of the vertices of $e\cup f$, for all twin pairs of edges $\{e,f\}$. The relative order of the vertices within an edge plays an important role important role in palettes, which are used to study the uniform Tur\'an density of $3$-graphs; the notions that we present here are inspired by palettes.

Given a set $S$ of $k$ elements, a pair of elements $\{u,v\}\in S$, an ordering $\preceq$ of $S$ and a pair of numbers $\{i,j\}\in\binom{[k]}{2}$, we say that $\{u,v\}$ play the role $\{i,j\}$ in $S$ if, when the elements of $S$ are sorted as $w_1\prec w_2\prec\dots\prec w_k$, we have $\{w_i,w_j\}=\{u,v\}$. We will use this notion most commonly in graphs with a vertex ordering, in order to describe the role of a pair of vertices within an edge.

We say that a quasi-linear hypergraph $H$ with a linear vertex order $\preceq$ is \emph{consistently ordered} if for every twin pair of edges $\{e,f\}$, the pair $e\cap f$ plays the same role in $e$ as in $f$. We say that $H$ is \emph{consistent} if there exists a vertex order on which it is consistently ordered, and we say that $H$ is \emph{inconsistent} otherwise.

We will introduce a way of describing the relative position of all vertices in a twin pair of edges in a vertex-ordered quasi-linear graph. A \emph{descriptive sequence} of order $k$ is a sequence $\sigma=(s_1, s_2, \dots, s_{2k-2})$ of $2k-2$ letters from $\{X,Y,Z\}$, containing $k-2$ letters $X$, $k-2$ letters $Y$ and two letters $Z$. We say that a descriptive sequence is \emph{consistent} if the pair of $Z$ entries plays the same role in the set of $k$ entries $\{X,Z\}$ as in the set of $k$ entries $\{Y,Z\}$. We say that the descriptive sequence is \emph{inconsistent} otherwise. For example, if $k=5$, the sequence $XYYXZZYX$ is consistent, because the pair of $Z$ entries play the role $\{3,4\}$ among both the $\{X,Z\}$ and the $\{Y,Z\}$ entries. On the other hand, the sequence $YXZXYXZY$ is inconsistent, since the pair of $Z$ entries plays the role $\{2,5\}$ among the $\{X,Z\}$ entries and the role $\{2,4\}$ among the $\{Y,Z\}$ entries.

Given two edges $e,f$ with $|e\cap f|=2$ in a vertex-ordered graph $H$, we say that the descriptive sequence $\sigma=(s_1, s_2, \dots, s_{2k-2})$ describes the pair $\{e,f\}$ if, when the vertices of $e\cup f$ are sorted as $w_1\prec w_2\prec\dots\prec w_{2k-2}$, the indices of the vertices of $e$ correspond to the indices of the entries $\{X,Z\}$ or the entries $\{Y,Z\}$, and the same holds for the indices of the vertices of $f$. In other words, one of the following situations hold:

\[s_i=\left\{\begin{array}{cc} X & \text{if }w_i\in e\setminus f\\ Y & \text{if }w_i\in f\setminus e\\ Z & \text{if }w_i\in e\cap f\end{array}\right.\quad\text{or}\quad\left\{\begin{array}{cc} X & \text{if }w_i\in f\setminus e\\ Y & \text{if }w_i\in e\setminus f\\ Z & \text{if }w_i\in e\cap f\end{array}\right.\]

We say that a quasi-linear hypergraph $H$ admits the descriptive sequence $\sigma$ if there exists a linear order of its vertex set in which every twin pair of edges $\{e,f\}$ is described by $\sigma$.

A particularly interesting descriptive sequence is $XX\dots XZZYY\dots Y$. Given a graph with an ordered vertex set, we will call the first two vertices of an edge its \emph{head}, and the last two vertices its \emph{tail}. Observe that here the words \emph{head} and \emph{tail} each refers to a pair of vertices, rather than a single vertex. Then the sequence $XX\dots XZZYY\dots Y$ describes a pair of edges if and only if the head of one edge is the tail of the other.

We say that a hypergraph is \emph{head-tail-mixing} if for every linear order of its vertex set there exist edges $e,f$ such that the head of $e$ is the tail of $f$.

We now have all the required definitions to characterize the families of hypergraphs in which our theorems apply

\begin{theorem}\label{thm:dens14} For $r\geq 3$ every quasi-linear, head-tail-mixing $r$-graph $F$ which admits the descriptive sequence $XX\dots XZZYY\dots Y$ satisfies $\pi_u(F)=1/4$.\end{theorem}

\begin{theorem}\label{thm:denspi} For $r\geq 3$ every quasi-linear, inconsistent $r$-graph $F$ which admits all inconsistent descriptive sequences of order $r$ satisfies $\pi_u(F)=\pi_r$.\end{theorem}

Of course, in order to be able to say that $1/4$ and $\pi_r$ are uniform Tur\'an density values we need the corresponding existence theorems:

\begin{theorem}\label{thm:exist14} For every $r\geq 3$ there exists a quasi-linear, head-tail-mixing $r$-graph $F$ which admits the descriptive sequence $XX\dots XZZYY\dots Y$,.\end{theorem}

\begin{theorem}\label{thm:existpi}For every $r\geq 3$ there exists a quasi-linear, inconsistent $r$-graph $F$ which admits all inconsistent descriptive sequences of order $r$.\end{theorem}

These four theorems together directly imply Theorem~\ref{thm:main14} and Theorem~\ref{thm:mainpi}.

\section{Proof of existence theorems}\label{sec:exist}

In this section we will prove Theorems~\ref{thm:exist14} and~\ref{thm:existpi}. The proofs of both theorems follow a similar structure, but the former is simpler, since at a certain step it uses the Erd\H{o}s-Szekeres theorem while the proof of Theorem~\ref{thm:existpi} uses Lemma~\ref{lem:FisGra}. In both cases the construction of the $r$-graph is obtained by starting with the sequence of nowhere-empty $2r-2$-graphs from Lemma~\ref{lem:nowhere}, and subdividing some or all edges into two edges of size $r$ intersecting in two vertices. This process is once again simpler in the proof of Theorem~\ref{thm:exist14} than Theorem~\ref{thm:existpi}.

\begin{proof}[Proof of Theorem~\ref{thm:exist14}] Let $\{H'_n\}$ be a nowhere-empty sequence of $2r-2$-graphs, whose existence is guaranteed by Lemma~\ref{lem:nowhere}. We can assume that the vertex set of $H'_n$ is $[n]$. Construct the $r$-graph $H_n$ from $H'_n$ using the following procedure. The vertex set is $[n]$. For every edge $e\in E(H'_n)$, if $v_1<v_2<\dots<v_{2r-2}$ are the vertices of $e$, then add to $H_n$ the two edges $v_1v_2\dots v_r$ and $v_{r-1}v_r\dots v_{2r-2}$.

Since the hypergraph $H'_n$ is linear, the only pairs of edges of $H_n$ that intersect in two vertices are those that are generated by the same edge of $H'_n$, otherwise the intersection size is at most 1. Thus $H_n$ is quasi-linear, and the twin pairs are precisely the pairs of edges originating from the same edge of $H'_n$. Moreover, in the natural order of $[n]$, in every twin pair of edges, the head of one edge is the tail of the other, therefore $H_n$ admits the descriptive sequence $XX\dots XZZYY\dots Y$. To conclude the proof of Theorem~\ref{thm:exist14}, we will prove that for some value of $n$ the hypergraph $H_n$ is head-tail-mixing.

Let $t=2r-2$, and let $n$ be a large enough multiple of $t^2$. Let $\preceq$ be a linear order on $[n]=V(H_n)$. Our goal is to show that, when sorting the vertices by $\preceq$, there exist two edges $e,f$ such that the head of $e$ is the tail of $f$. 

Divide $[n]$ into $t^2$ intervals of equal length, denoted by $I_1, I_2, \dots, I_{t^2}$. We apply Lemma~\ref{lem:IMO}, with $k=t^2$ and $A_i=I_i$, to find subsets $B_i\subset I_i$, each with size at least $|I_i|/t^2=n/t^4$, such that for $i\neq j$ there is no element of $B_i$ which lies between two elements of $B_j$ in the order $\preceq$. That means that if for some $x\in B_i$ and $y\in B_j$ we have $x\prec y$, then this holds for all choices of $x$ and $y$ within the sets $B_i$ and $B_j$. In this case we write $B_i\prec B_j$.

Consider the ordering that $\preceq$ induces on the sequence $B_1, B_2, \dots, B_{t^2}$. By the Erd\H{o}s-Szekeres theorem, there is a monotone subsequence $B_{i_1}, B_{i_2}, \dots, B_{i_t}$ of length $t$. On every set consisting of one vertex from each $B_{i_j}$, the order $\preceq$ is monotone.

Since the family of hypergraphs $H'_n$ is nowhere-empty, for $n$ large enough there exists an edge $e'\in E(H'_n)$ with one vertex in each set $B_{i_j}$. Consider the two edges $e,f\in E(H_n)$ that were generated from $e'$. In the natural order of the integers, w.l.o.g. the head of $e$ is the tail of $f$. But because the $\preceq$ order is monotone in $e'$, the order of its vertices in $\preceq$ is either the same as in the natural order or its reverse. In either case, the head of one of $e,f$ is the tail of the other, concluding the proof that $H_n$ is head-tail-mixing. \end{proof}

In the case of Theorem~\ref{thm:existpi}, there are some changes with respect to the proof above. In order to account for the many different vertex orders that yield each of the inconsistent descriptive sequences, we will consider that the vertex set of our hypergraph forms a multidimensional grid, and that each descriptive sequence is obtained by sorting the vertices by one of its coordinates. This multidimensionality of the vertex set will mean that, instead of the relatively simple Erd\H{o}s-Szekeres theorem, we will use Lemma~\ref{lem:FisGra}. In addition, not all edges of $H'_n$ will become two edges of $H_n$. The reason for this is that, if we do split all edges, it might not be possible for the new twin pair of edges to be described by the correct descriptive sequence when sorting by each of the coordinates.

\begin{proof}[Proof of Theorem~\ref{thm:existpi}]
Let $t=2r-2$. Let $d$ be the number of inconsistent descriptive sequences of order $r$, and let $\sigma_1, \sigma_2, \dots, \sigma_d$ be a list containing each of them. Set $N=N(t,d)$ as in Lemma~\ref{lem:FisGra}. We set $m$ to be a large enough integer divisible by $N$. Finally, set $n=m^d$.

Consider the nowhere-empty sequence of linear $t$-graphs $\{H'_i\}_{i=1}^\infty$. We identify the vertices of $H'_n$ with the elements of $[m]^d$. We will form $H_n$ from $H'_n$ by turning some (but not all) edges from $H'_n$ into twin pairs of edges in $H_n$. We now explain the procedure for deciding whether to split an edge and how to do so.

Let $e$ be an edge of $H'_n$. If there exists a value $i\in[d]$ such that two vertices of $e$ have the same value of their $i$-th coordinate, discard this edge. Therefore, we will assume that for every $i$, the $i$-th coordinate of all vertices of $e$ is different.

Remember that each descriptive seqence $\sigma_i$ is a sequence of $t=2r-2$ symbols, out of which $r-2$ are $X$, $r-2$ are $Y$ and two are $Z$. Suppose that it is possible to label $r-2$ vertices of $e$ by $X$, $r-2$ vertices by $Y$ and two by $Z$ in such a way that, for every $i\in [d]$, when the vertices of $e$ are sorted by their $i$-th coordinate, the resulting label sequence is precisely $\sigma_i$. In that case, add to $H_n$ two edges, one by taking the vertices labeled $X$ and $Z$, and the other by taking the vertices labeled $Y$ and $Z$. If no such labeling exists, discard this edge.

Since the hypergraph $H'_n$ is linear, the only pairs of edges of $H_n$ that intersect in two vertices are those that are generated by the same edge of $H'_n$, otherwise the intersection size is at most 1. Thus $H_n$ is quasi-linear, and the twin pairs of edges are precisely the pairs of edges originating from the same edge of $H'_n$. Moreover, for every inconsistent desctiptive sequence $\sigma_i$, when the vertices of $H_n$ are sorted by their $i$-th coordinate, every twin pair of edges is described by $\sigma_i$, meaning that $H_n$ admits all inconsistent descriptive sequences. To conclude the proof of Theorem~\ref{thm:existpi}, we will show that for some $n$ the hypergraph $H_n$ is inconsistent.

Let $\preceq$ be a linear order on the vertices of $H_n$. Our goal is to show that there exist edges $e$ and $f$, with intersection size two, such that $e\cap f$ plays a different role in $e$ and in $f$. 

Subdivide $[m]$ into $N$ equal intervals, denoted $I_1, \dots, I_N$ from smallest to largest. For each vector $\vec a=(a^1, \dots, a^d)\in [N]^d$, we can write $I_{\vec a}=I_{a^1}\times\dots\times I_{a^d}$. These sets $I_{\vec a}$ form a partition of the vertex set of $H'_n$, and each has size $(m/N)^d=n/N^d$.

We apply Lemma~\ref{lem:IMO} to the $N^d$ sets $I_{\vec a}$ to find subsets $B_{\vec a}\subseteq I_{\vec a}$, each with size at least $|B_{\vec a}|\geq |I_{\vec a}|/N^d=n/N^{2d}$, such that for every $\vec a\neq \vec b$ no element of $B_{\vec a}$ lies between two elements of $B_{\vec b}$ in the ordering $\preceq$. Because of this, $\preceq$ induces a linear order on the family of sets $B_{\vec a}$.

If we identify the set $B_{\vec a}$ with its index $\vec a$, we can apply Lemma~\ref{lem:FisGra} to this ordering to find sets $S_1, \dots, S_d$, each with size $t$, such that the family of sets $I_{\vec a}$ with  $\vec a\in S_1\times \dots \times S_t$ is sorted by layers. W.l.o.g. let us assume that the family of sets is sorted or inversely sorted by the first coordinate, as all other cases are analogous.

Next we define vectors $\vec X_1, \dots, \vec X_{r-2}, \vec Y_1, \dots, \vec Y_{r-2}, \vec Z_1, \vec Z_2\in [N]^d$. For each symbol $W\in \{X, Y, Z\}$, each index $j$ and each $i\in [d]$, let $W_j^i\in[t]$ be the position of the $j$-th letter $W$ in the descriptive sequence $\sigma_i$. Let the $i$-th coordinate in $\vec W_j$ be the $W_j^i$-th smallest element of $S_i$. Notice that all vectors $\vec w_j$ lie in $S_1\times\dots\times S_d$, and when sorted by their $i$-th coordinate, the letter sequence produced is precisely $\sigma_i$.

Because the sequence $\{H'_n\}_{n=1}^\infty$ is nowhere-empty, and $|B_{\vec a}|\geq n/N^{2d}$, for $n$ large enough there exists an edge $e'\in E(H'_n)$ which has one vertex in each of the sets $B_{\vec X_1}, \dots B_{\vec X_{r-2}}, B_{\vec Y_1}, \dots, B_{\vec Y_{r-2}}, B_{\vec Z_1}, B_{\vec Z_2}$. If we label each vertex of $e'$ with the letter in $\{X,Y,Z\}$ of the corresponding set, then for every $i\in[d]$, when we sort the vertices of $e'$ by the $i$-th coordinate, the resulting letter sequence is $\sigma_i$. That means that $e'$ contains a twin pair of edges of $H_n$, namely those with labels $X,Z$ and $Y,Z$.

We return to the vertex ordering $\preceq$. Remember that the sets $B_{\vec a}$ with $\vec a\in S_1\times \dots\times S_d$ are either sorted or inversely sorted by the first coordinate. As a consequence, the vertices of $e'$ are either in increasing or decreasing order of their first coordinate. This means that either the inconsistent descriptive sequence $\sigma_1$ or its reverse describe the twin pair $e,f$. In either case the role of $e\cap f$ in each of the edges $e$ and $f$ with the ordering $\preceq$ is different, as we wanted to show. We conclude that $H_n$ is inconsistent, finishing the proof of Theorem~\ref{thm:existpi}.
\end{proof}

\section{Reduced hypergraphs}\label{sec:reduced}

In the proof of our existence theorems we will use auxiliary hypergraphs called $(k,r)$-reduced graphs. These graphs are a generalization of the $n$-reduced hypergraphs introduced by Reiher in~\cite{Rei20}, which played a crucial role in many proofs related to uniform Tur\'an density. Their importance lies on their relation to the partitions obtained in the hypergraph regularity lemma.

A $(k,r)$-reduced hypergraph is an $\binom{r}{2}$-graph $H$ whose vertex set is partitioned into $\binom{k}{2}$ parts $V_{i,j}$, with $1\leq i<j\leq k$. For simplicity of writing we will not care about the order of the subindices, so $V_{i,j}$ and $V_{j,i}$ will denote the same set. For each edge $e\in E(H)$ there exist $r$ indices $t_1, t_2, \dots, t_r$, such that $e$ contains exactly one vertex in each set $V_{t_i,t_j}$ for each $1\leq i<j\leq r$.

Given indices $1\leq t_1<t_2<\dots<t_r\leq k$, the constituent $\mathcal{A}_{t_1, t_2, \dots, t_r}$ is the $\binom{r}{2}$-partite $\binom{r}{2}$-graph induced on the vertex sets $V_{t_i,t_j}$ with $1\leq i<j\leq r$. We say that $H$ has density at least $d$ if all constituents have edge density at least $d$.

We will briefly describe the connection between reduced graphs and hypergraph regularity, although it will necessitate the use of notation that will be introduced in Section~\ref{sec:reg}. These reduced graphs can be seen as the result of contracting each class of the partition $\fP^{(2)}$ of the pairs of vertices of a large $r$-graph $H'$ into a single vertex. If the edge density in $H'$ between any $r$ parts of $\fP^{(1)}$ is greater than $d$, then the resulting $(k,r)$-reduced hypergraph has density at least $d$. This will be the case if $H'$ is taken from a locally $d+\varepsilon$-dense sequence, after some cleanup (see Step 3 from the proof of Theorem~\ref{thm:dens14} and Theorem~\ref{thm:denspi}).

Our main lemmas involving $(k,r)$-reduced hypergraphs will have a connection with descriptive sequences of order $r$. Let $\sigma=(s_1, s_2, \dots, s_{2r-2})$ be a descriptive sequence. Let $H$ be a $(k,r)$-reduced hypergraph, and let $t_1<t_2<\dots<t_{2r-2}$ be indices. Let $x_1<x_2<\dots<x_r$ be the indices $t_i$ whose corresponding letter $s_i$ is $X$ or $Z$, and let $y_1<y_2<\dots<y_r$ be the indices whose corresponding letter is $Y$ or $Z$. We say that the $2r-2$-tuple $\{t_1, t_2, \dots, t_{2r-2}\}$ admits the descriptive sequence $\sigma$ if there exist edges $e_X\in\mathcal{A}_{x_1, x_2, \dots, x_r}$ and $e_Y\in \mathcal{A}_{y_1, y_2, \dots, y_r}$ intersecting in one vertex. The common vertex is in $V_{t_i,t_j}$, where $i$ and $j$ are the values such that $s_i=s_j=Z$.

We can now state our main lemmas on reduced hypergraphs. These lemmas will be used in the proofs of Theorem~\ref{thm:dens14} and Theorem~\ref{thm:denspi} to handle the analysis of the second layer of the regularity partition.

\begin{lemma}\label{lem:red14}
    For every $\varepsilon>0$, $r$ and $m$ there exists $k$ with the following property: if $H$ is a $(k,r)$-reduced hypergraph with density at least $1/4+\varepsilon$, then there exists a subset $S\subseteq [k]$ of $m$ indices such that every $2r-2$-tuple of indices in $S$ admits the descriptive sequence $XX\dots XZZYY\dots Y$.
\end{lemma}

\begin{lemma}\label{lem:redpi}
    For every $\varepsilon>0$, $r$ and $m$ there exists $k$ with the following property: if $H$ is a $(k,r)$-reduced hypergraph with density at least $\pi_r+\varepsilon$, then there exists a subset $S\subseteq [k]$ of $m$ indices and an inconsistent descriptive sequence $\sigma$ such that every $2r-2$-tuple of indices in $S$ admits $\sigma$.
\end{lemma}

\begin{proof}[Proof of Lemma~\ref{lem:red14}]
 We can assume that $m\geq 2r-2$, otherwise the statement holds trivially for $k=m$. Let $q=\lceil2/\ep\rceil$. For every $r$-tuple of indices $t_1<t_2<\dots<t_r$, we define its signature as follows. Let $W$ be the set of vertices which are incident to at least one edge in $\mathcal{A}_{t_1, t_2, \dots, t_r}$. Then the signature of $\{t_1, t_2, \dots, t_r\}$ is the pair \[\left(\left\lfloor q\cdot\frac{|V_{t_1,t_2}\cap W|}{|V_{t_1, t_2}|}\right\rfloor,\left\lfloor q\cdot\frac{|V_{t_{r-1},t_r}\cap W|}{|V_{t_{r-1}, t_r}|}\right\rfloor\right).\]

 Note that there are $(q+1)^2$ possible signatures. If $k=R_r(m, (q+1)^2)$, there exists a set $S$ of $m$ indices in which every $r$-tuple of indices has the same signature $(a,b)$.

 Next we claim that $a+b>q$. Let $t_1<t_2<\dots<t_r$ be $r$ indices in $S$, and let $W$ be as above. Because each edge in $\mathcal{A}_{t_1, t_2, \dots, t_r}$ has a vertex in each $V_{t_i,t_j}$ with $1\leq i<j\leq r$, and that vertex must be in $W$ by definition, the number of edges in $\mathcal{A}_{t_1, t_2, \dots, t_r}$ is at most $\prod_{1\leq i<j\leq r}|V_{t_i, t_j}\cap W|$. On the other hand, because $H$ has density at least $1/4+\varepsilon$, the number of edges in $\mathcal{A}_{t_1, t_2, \dots, t_r}$ is at least $(1/4+\varepsilon)\prod_{1\leq i<j\leq r}|V_{t_i, t_j}|$. Dividing the two equations we obtain that \[\frac14+\ep\leq\prod\limits_{1\leq i<j\leq r}\frac{|V_{t_i,t_j}\cap W|}{|V_{t_i,t_j}|}\leq \frac{|V_{t_1,t_2}\cap W|}{|V_{t_1,t_2}|}\cdot \frac{|V_{t_{r-1},t_r}\cap W|}{|V_{t_{r-1},t_r}|}\leq \frac{a+1}{q}\cdot \frac{b+1}{q}.\]
 From here we conclude $(a+b+2)^2\geq (1+4\varepsilon)q^2$, so $a+b> (1+\varepsilon)q-2\geq q$.

 We are now ready to prove that $S$ satisfies the property in our statement. Let $t_1<t_2<\dots<t_{2r-2}$ be indices in $S$. Let $W_1$ and $W_2$ be the set of vertices in $V_{t_{r-1},t_r}$ which are incident to an edge in $\mathcal{A}_{t_1, t_2, \dots, t_r}$ and in $\mathcal{A}_{t_{r-1}, t_r, \dots, t_{2r-2}}$, respectively. Then \[|W_1\cap W_2|\geq |W_1|+|W_2|-|V_{t_{r-1}, t_r}|\geq \frac{a+b-q}{q}|V_{t_{r-1}, t_r}|>0.\] There exist edges $e_1\in \mathcal{A}_{t_1, t_2, \dots, t_r}$ and $e_2\in \mathcal{A}_{t_{r-1}, t_r, \dots, t_{2r-2}}$ intersecting in a vertex, so $\{t_1, t_2, \dots, t_{2r-2}\}$ admits the descriptive sequence $XX\dots XZZYY\dots Y$.\end{proof}

 \begin{proof}[Proof of Lemma~\ref{lem:redpi}]

Let $q=\left\lceil 2^{\binom{r}{2}}/\varepsilon\right\rceil$, let $p=R_r\left(r^3, (q+1)^{\binom{r}{2}}\right)$ and let $k=R_{2r-2}(\max\{p,m\}, 3^{2r-2}+1)$. Let $H$ be a $(k,r)$-reduced hypergraph. Consider each $2r-2$-tuple of indices in $[k]$. If it admits one or more inconsistent descriptive sequences, assign one arbitrarily, otherwise assign it the empty sequence. By the choice of $k$, there is a subset $S$ of $\max\{p, m\}$ indices where every $2r-2$-tuple gets assigned the same sequence. If this sequence is non-empty, we are done. We will assume that they get assigned the empty sequence, and reach a contradiction.

For each $r$-tuple of indices $t_1<t_2<\dots<t_r$ in $S$, we define its profile $(a_{1,2}, a_{1, 3}, \dots, a_{r-1, r})$, where $a_{i,j}=\lfloor q\cdot |V_{t_i, t_j}\cap W|/|V_{t_i, t_j}| \rfloor$ and $W$ is the set of vertices incident to at least one edge in $\mathcal{A}_{t_1, t_2, \dots, t_r}$. Note that there are at most $(q+1)^{\binom{r}{2}}$ distinct profiles. By the choice of $p$, there exists a subset $S'\subseteq S$ of $q^3$ indices where every $r$-tuple has the same profile $(a_{1,2}, a_{1, 3}, \dots, a_{r-1, r})$.

We claim that $\sum_{1\leq i<j\leq r}a_{i,j}>1$. To see this, let $t_1<t_2<\dots<t_r$ be indices in $S'$, and let $W$ as above. On the one hand, the number of edges in $\mathcal{A}_{t_1, t_2, \dots, t_r}$ is at least $(\pi_r+\varepsilon)\prod_{1\leq i<j\leq r}|V_{t_i, t_j}|$, since $H$ has density at least $\pi_r+\varepsilon$. On the other hand, the number of edges is at most $\prod_{1\leq i<j\leq r}|V_{t_i, t_j}\cap W|$. Since $\frac{|V_{t_i,t_j}\cap W|}{|V_{t_i, t_j}|}\leq \frac{a_{i,j}+1}{q}$, by the arithmetic-geometric mean inequality we have 

\begin{align*}
\sum\limits_{1\leq i<j\leq r}a_{i,j}=&\binom{r}{2}\left(q\left(\binom{r}{2}^{-1}\sum\limits_{1\leq i<j\leq r}\frac{a_{i,j}+1}{q}\right)-1\right)\\ \geq& \binom{r}{2}\left(q\left(\prod\limits_{1\leq i<j\leq r}\frac{a_{i,j}+1}{q}\right)^{\frac{1}{\binom{r}{2}}}-1\right)\\ \geq&\binom{r}{2}\left((\pi_r+\varepsilon)^{\frac{1}{\binom{r}{2}}}-1/q\right)\\ =&\left(1+\frac{\ep}{\pi_r}\right)^{\frac{1}{\binom{r}{2}}} -\frac{\binom{r}{2}}{q}>1.\end{align*}

Within $S'$, choose elements $a,b\in S'$ and $r$-tuples of indices $R_{1, 2}, R_{1, 3}, \dots, R_{{r-1}, r}$ with the following properties:
\begin{itemize}
    \item $a,b\in R_{i,j}$ for all $1\leq i<j\leq r$.
    \item The sets $R_{i,j}\setminus \{a,b\}$ are pairwise disjoint.
    \item For all $1\leq i<j\leq r$, if $t_1<t_2<\dots<t_r$ are the elements of $R_{i,j}$, then $t_i=a$ and $t_j=b$.
\end{itemize}

One way to construct these sets is to start with a family consisting of $\binom{r}{2}$ sets $R_{i,j}$, intersecting in two elements $a,b$ and otherwise disjoint. Define a partial order $\preceq$ on the union of these sets: for each $1\leq i<j\leq r$, let $i-1$ elements of $R_{i,j}\setminus\{a,b\}$ be smaller than $i$, $j-i-1$ of these elements be larger than $i$ and smaller than $j$, and the remaining $r-j$ elements be larger than $b$. This forces the pair $\{a,b\}$ to play the role $\{i,j\}$ in $R_{i,j}$ on any linear extension of $\preceq$. Pick arbitrarily any such linear extension, and identify the elements of $\cup_{i,j}R_{i,j}$ with elements of $S'$ in the right order of $\preceq$. 

Let $W_{i,j}$ be the set of vertices in $V_{a,b}$ incident to at least one edge in $\mathcal{A}_{R_{i,j}}$. We have that $|W_{i,j}|\geq a_{i,j}|V_{a,b}|$, so the sizes of all $W_{i,j}$ add up to more than $|V_{a,b}|$. This means that the sets are not disjoint, and there exist $(i,j)\neq(i',j')$ such that $W_{i,j}\cap W_{i', j'}\neq\emptyset$. There exist edges $e_X\in \mathcal{A}_{R_{i,j}}$ and $e_Y\in \mathcal{A}_{R_{i', j'}}$ intersecting at a vertex of $V_{a,b}$.

Let $t_1<t_2<\dots<t_{2r-2}$ be the elements of $R_{i,j}\cup R_{i', j'}$. Let $\sigma=(s_1, s_2, \dots, s_{2r-2})$ be the descriptive sequence where $s_i=X$ if $t_i\in R_{i,j}\setminus R_{i', j'}$, $s_i=Y$ if $t_i\in R_{i',j'}\setminus R_{i,j}$ and $s_i=Z$ if $t_i\in R_{i,j}\cap R_{i', j'}=\{a,b\}$. This descriptive sequence is inconsistent, because the pair of $Z$ entries plays the role $\{i,j\}$ among the entries $\{X,Z\}$ and the role $\{i',j'\}$ among the entries $\{Y,Z\}$.

As a consequence, the $2r-2$-tuple of indices $t_1, t_2, \dots, t_{2r-2}$ admits the inconsistent descriptive sequence $\sigma$. This contradicts the fact that every $2r-2$-tuple of indices in $S$ was assigned the empty sequence, finishing the proof.
\end{proof}

\section{Hypergraph regularity}\label{sec:reg}

The proofs of Theorems~\ref{thm:dens14} and~\ref{thm:denspi} are based on the use of the regularity lemma, and its corresponding counting lemma. Throughout the existing literature one can find a multitude of versions and variations of the hypergraph regularity lemma, not all of them equivalent to each other. The formulation used here, as well as all the notation involved, is taken from a survey by R\"odl and Schacht~\cite{RodS07, RodS07a}. Like most versions of the hypergraph regularity lemma, the statement requires introducing a number of concepts and definitions. The purpose of this section is to introduce and state the hypergraph regularity lemma and the counting lemma, in which we will only care about the number of copies of a certain subgraph being non-zero.

The first concept that we need to introduce is vertex partitions. Let $V$ be a vertex set. Unlike the graph regularity lemma, where we only need to partition $V$ itself, when we apply the regularity lemma to $r$-graphs, we need additional partitions of sets of pairs of vertices, triples of vertices, and so on up to $(r-1)$-tuples of vertices. These partitions must be nested, in a sense that we will explain soon.

Generally, we will denote partitions with Fraktur letters such as $\mathfrak{P}$, and the parts in those partitions with caligraphic letters such as $\mathcal{P}$. We will write a superindex between parenthesis to indicate the number of vertices in each element of the partitioned set: for example, $\mathfrak{P}^{(3)}$ denotes a partition of a set of triples of vertices.

Let $\mathfrak{P}^{(1)}=\{\mathcal{P}^{(1)}_1, \dots, \mathcal{P}^{(1)}_k\}$ be a partition of the vertex set $V$, and let $i$ be a positive integer. We denote the set of all $i$-tuples of vertices of $V$, with at most one vertex from each part of $\mathfrak{P}^{(1)}$, by $\cross_i(\fP^{(1)})$. 

Let $\fP^{(1)}$ be a partition of $V$, and for each $2\leq i\leq r-1$, let $\fP^{(i)}$ be a partition of $\cross_i(\fP^{(1)})$. Let $\fP=\{\fP^{(1)}, \fP^{(2)}, \dots, \fP^{(r-1)}\}$. Given a set $I\in\cross_i(\fP^{(1)})$, we denote the part of $\fP^{(i)}$ that contains $I$ by $\cP^{(i)}(I)$. We say that $\fP$ is a \emph{nested} family of partitions if the following holds: for every $2\leq i\leq r-1$, for every $\cP^{(i)}\in\fP^{(i)}$ and every $\cP^{(i-1)}\in\fP^{(i-1)}$, if some element of $\cP^{(i)}$ is a superset of some element of $\cP^{(i-1)}$, then every element of $\cP^{(i)}$ is a superset of some element of $\cP^{(i-1)}$. In other words, if we fix a part $\cP^{(i)}\in\fP^{(i)}$, pick some $T\in \cP^{(i)}$, and then consider the family of parts $\{\cP^{(i-1)}(T-v)|v\in T\}$, that family does not depend on the choice of $T$.

The concept of \emph{polyad} will be useful to define how elements of one layer of $\fP$ are placed with respect to another layer. Given $j<i$, given a set $I\in\cross_i(\fP^{(1)})$, we define the $j$-th polyad of $I$, denoted by $\hat \cP^{(j)}(I)=\{\cP^{(j)}(J)|J\subset I, |J|=j\}$. Observe that, by the condition above, in a nested family of partitions, given a class $\cP^{(i)}\in \fP^{(i)}$, the $i-1$-th polyad of all elements of $\cP^{(i)}$ is the same, and by induction, so is the $j$-th polyad for any $j<i$. We define $\hat\fP^{(j)}$ to be the set of all $j$-th polyads of elements of $\cross_{j+1}(\fP^{(1)})$.

Given a vector $\vec a=(a_1, \dots, a_r)$, we say that the family of partitions $\fP=\{\fP^{(1)}, \dots, \fP^{(r)}\}$ is $(r,\vec a)$-nested if $\fP^{(1)}$ consists of $a_1$ parts, and for each $2\leq i\leq r$, for each $\hat\cP^{(i-1)}\in\hat\fP^{(i-1)}$, there are exactly $a_i$ classes in $\fP^{(i)}$ such that the $i-1$-th polyad of their elements is $\hat\cP^{(i-1)}$. Moreover, we say that $\fP$ is $t$-bounded if all entries of $\vec a$ are less than or equal to $t$.

In the regularity lemma for graphs, we require that the partition is equitable, that is, all the parts have the same number of vertices. Something similar will hold here, except the condition will be more complicated. In a nutshell, we want each element of $\hat\fP^{(i-1)}$ to be the $(i-1)$-th polyad corresponding to roughly the same number of elements of $\cross_i(\fP^{(1)})$. In order to guarantee this, we need to start discussing the regularity conditions that the final partition will satisfy.

Let $i\leq k-1$, and let $H$ and $H'$ be an $i-1$-graph and an $i$-graph, respectively, on the same vertex set. We denote  by $K_i(H)$ the set of $i$-tuples of vertices $I$ such that $I-v\in H$ for all $v\in I$. We call this the set of $i$-cliques of $H$. In the particular case in which $i=2$, and $H'$ is a bipartite graph with a given partition $U,V$ (which is the case for each of the parts of $\fP^{(2)}$), we will define $K_2(H)$ as only containing the edges from $H\cap U$ to $H\cap V$, rather than all pairs of vertices in $H$ (observe that in this case $H$ is a $1$-graph, that is, a collection of vertices).

The \emph{density} of $H'$ w.r.t. $H$, denoted by $d(H'|H)$, as the value \[\frac{|H'\cap K_i(H)|}{|K_i(H)|},\] that is, the proportion of the $i$-cliques of $H$ which are edges of $H'$. If $H$ does not contain any $i$-clique, we take this density to be 0.

Given a collection of $\beta$ $(i-1)$-graphs $B=\{H_1, H_2, \dots, H_\beta\}$, we define $K_i(B)=\cup_{j=1}^\beta K_i(H_j)$. Then the density of $H'$ over $B$, denoted by $d(H'|B)$, can be defined analogously.
We say that $H'$ is $(d,\xi,\beta)$-regular w.r.t. $H$ if the following holds for all families $B$ of $\beta$ subgraphs $Q_1, Q_2, \dots, Q_\beta\subseteq H$, if $|K_i(B)|\geq\xi|K_i(H)|$, then $\left|d(H'|B)-d\right|\leq\xi$.

Let $\fP$ be a $(r-1,\vec a)$-nested family of partitions, and suppose that $a_1$ divides $|V|$. Let $\eta$ be a positive real number, and $\vec\xi=(\xi_2, \xi_3, \dots, \xi_{r-1})$ be a positive vector. We say that $\fP$ is $(\eta, \vec \xi, \vec a,\beta)$-equitable if the following three conditions hold:
\begin{itemize}
\item $\left|\cross_k(\mathfrak{P}^{(1)})\right|\geq (1-\eta)\binom{n}{k}$.
\item All parts of $\mathfrak{P}^{(1)}$ have equal size.
\item For every $2\leq j\leq r-1$, for every $\mathcal{P}^{(j)}\in\mathfrak{P}^{(j)}$, we have that $\mathcal{P}^{(j)}$ is $(1/a_j,\xi_j,\beta)$-regular in its polyad $\hat{\mathcal{P}}^{(j-1)}(\mathcal{P}^{(j)})$.
\end{itemize}

We are almost ready to state the hypergraph regularity lemma. All that is left to do is define when a hypergraph is regular w.r.t. a partition. We say that $H'$ is $(*,\xi,\beta)$-regular w.r.t. $H$ if it is $(d,\xi,\beta)$-regular for some value of $d$.

\begin{definition}[$(*,\xi,\beta)$-regular w.r.t. $\mathfrak{P}$] Let $\xi$ be a positive real number and $\beta$ a positive integer. Suppose that $H$ is an $r$-uniform hypergraph on vertex set $V$, and $\mathfrak{P}$ is an $(r-1,\vec a)$-nested family of partitions. We say that $H$ is $(*,\xi,\beta)$-regular w.r.t. $\mathfrak{P}$ if the number of elements $K\in \cross_r(\mathfrak{P}^{(1)})$ such that $H$ is not $(*,\xi,\beta)$-regular w.r.t. the polyad $\hat{\mathcal{P}}^{(r-1)}(K)$ is at most $\xi|V(H)|^r$.\end{definition}

\begin{theorem}[Hypergraph regularity lemma, simplified from~\cite{RodS07}]\label{thm:regularity}
Let $r\geq 2$ be a fixed integer. For all positive constants $\eta$ and $\xi_r$, all functions $\beta:\NN^{r-1}\rightarrow\NN$, and functions $\xi_2, \xi_3, \dots, \xi_{r-1}$, with $\xi_i:\NN^{r-i}\rightarrow(0,1]$, there exist integers $t$ and $n_0$ for which the following holds:

For every $r$-uniform hypergraph $H$ on $n$ vertices, where $n\geq n_0$, where $t!$ divides $n$, there exists an $(r-1, \vec a)$-nested family of partitions $\fP$, for some vector $\vec a$, such that:

\begin{itemize}
\item $\fP$ is $(\eta, \vec\xi(\vec a), \vec a, \beta(\vec a))$-equitable and $t$-bounded, and
\item $H$ is $(*,\xi_r,\beta(\vec a))$-regular w.r.t. $\fP$.
\end{itemize}

where $\xi_i(\vec a)=\xi_i(a_i, a_{i+1}, \dots, a_{r-1})$.
\end{theorem}

The counterpart for the regularity lemma is the counting lemma. In this case, since we are only concerned with the existence of a copy of a certain hypergraph $F$ as a subgraph of $H$, rather than with approximating the number of copies, we can simplify the statement significantly. Moreover, we will state it for a case in which $H$ comes equipped with a nested family of partitions.

The key ingredient of the hypergraph counting lemma is defining the parts of the partition $\fP$ that will contain not just each vertex, but also each pair of vertices, each triple of vertices, and so on up to each $r-1$-tuple of vertices. We need to ensure that these choices are consistent: if $A\subseteq B$ are two sets of vertices, we need to make sure that the part of $\fP^{(|A|)}$ containing $A$ is in the polyad of the part of $\fP^{(|B|)}$ containing $B$.

Given an $r$-graph $F$, we denote by $F_{<i>}$ the set of $i$-tuples of vertices which are contained in some edge of $F$. In particular, $F_{<1>}=V(F)$ and $F_{<r>}=E(F)$.

\begin{theorem}[Counting lemma]\label{thm:count} For every $1\leq r\leq \ell$, every $r$-uniform hypergraph $F$ on vertex set $[\ell]$, for all $d_r$ there exists $\xi_r$ such that for all $a_{r-1}$, there exists $\xi_{r-1}$ such that for all $a_{r-2}$ there exists $\xi_{r-2}$ \dots\ such that for all $a_2$ there exist $\xi_2, \beta$ and $m_0$ such that the following holds:

Let $H$ be an $r$-uniform hypergraph, with a $(r-1,\vec a)$-nested, $(1,\vec \xi,\vec a, \beta)$-equitable family of partitions $\fP$, where the size of each part of $\fP^{(1)}$ is at least $m_0$. For each $1\leq i\leq r-1$, let $f_i:F_{<i>}\rightarrow \fP^{(i)}$ be a function, as well as $f_r:E(F)\rightarrow \hat\fP^{(r-1)}$. Suppose that the following conditions hold:
\begin{itemize}
\item For every $2\leq i\leq r-1$, for every $S\in F_{<i>}$, we have that $\hat\fP^{(i-1)}(f_i(S))=\cup_{v\in S}f_{i-1}(S-v)$.
\item For every $e\in E(F)$, we have $f_r(e)=\cup_{v\in e}f_{r-1}(e-v)$.
\item For every $e\in E(H)$, the hypergraph $H$ is $(d, \xi_r, \beta)$-regular w.r.t. $f_r(e)$, for some $d\geq d_r$.
\end{itemize}

Then $H$ contains $F$ as a subgraph.
\end{theorem}

Before starting with the proof, let us mention that, while we will only need the existence version of the \emph{hypergraph} counting lemma, we will use the more precise version of the \emph{graph} counting lemma, which we now state.

\begin{lemma}\label{lem:graphcount}
    Let $k$ be a positive integer. For every $\varepsilon>0$ there exists $\xi>0$ with the following property. Let $G$ be a graph and let $V_1, V_2, \dots, V_k$ be disjoint vertex subsets of size $n$. Let $S$ be the set of $k$-cliques in $G$ with one vertex in each $V_i$. If for every $1\leq i<j\leq k$ the bipartite graph induced on $(V_i,V_j)$ is $(d,\xi)$-regular, then $\left||S|-d^{\binom{k}{2}}n^k\right|\leq \varepsilon n^k$.
\end{lemma}

\section{Proof of density theorems}\label{sec:density}

\subsection{Overview of the proofs}

The proofs of Theorems~\ref{thm:dens14} and~\ref{thm:denspi} follow a similar outline. In order to prove the lower bounds on $\pu(F)$, we will construct a locally $1/4$-dense sequence of hypergraphs $H_n$ in which no pair of vertices is the head of some edge and the tail of another, or a locally $\pi_r$-dense sequence $H_n$ where every pair of edges intersecting in two vertices admits a consistent descriptive sequence. In either case, $F$ cannot be a subgraph of $H_n$.

Let $H_n$ be a locally $(1/4+\varepsilon)$-dense or $(\pi_r+\varepsilon)$-dense sequence of hypergraphs, respectively. Our goal is to show that $H_n$ contains $F$ as a subgraph.

\textbf{Step 1:} Apply the hypergraph regularity lemma to some large enough hypergraph $H$ in the sequence to obtain a regular family of partitions $\fP$. Denote the parts in $\fP^{(1)}$ as $V_1, \dots, V_q$. 

\textbf{Step 2:} Remove from $H$ all edges contained in polyads over which $H$ is not regular or where the density of $H$ w.r.t. the polyad is less than $\varepsilon/100$, and let $H'$ be the result. This guarantees that not too many edges are removed, and in particular only a tiny proportion of $r$-tuples of parts have decreased their density by more than $\varepsilon/50$. There exists a subset of $q_1$ parts, w.l.o.g. $V_1, \dots, V_{q_1}$, such that in every $r$-tuple of them the density of $H\setminus H'$ is at most $\varepsilon/50$. Restrict $H'$ and $\fP$ to these parts.

\textbf{Step 3:} There cannot be a large collection of vertex classes where the density of every $r$-tuple among them is less than $1/4+2\varepsilon/3$ or $\pi_r+2\varepsilon/3$, otherwise $H$ would not have the required local density. By Ramsey's theorem, there exists a subset of $q_2$ vertex sets, w.l.o.g. $V_1, \dots, V_{q_2}$, in which every $r$-tuple of them has density at least $1/4+2\varepsilon/3$ or $\pi_r+2\varepsilon/3$ in $H$, and density at least $1/4+\varepsilon/2$ or $\pi_r+\varepsilon/2$ in $H'$.

\textbf{Step 4:} Define an auxiliary $(q_2,r)$-reduced hypergraph $R$, where each vertex represents one part of $\fP^{(2)}$, and where an $\binom{r}{2}$-tuple of vertices forms an edge if they form the second polyad of an edge $e$ in $H'$. Then $R$ has density at least $1/4+\varepsilon/4$ or $\pi_r+\varepsilon/4$, respectively.

\textbf{Step 5:} Use Lemma~\ref{lem:red14} or Lemma~\ref{lem:redpi} to find a subset of $\ell=|V(F)|$ indices where every $2r-2$-tuple of indices in $R$ admits some descriptive sequence $\sigma=(s_1, s_2, \dots, s_{2r-2})$. In the case of Lemma~\ref{lem:red14} this sequence is $XX\dots XZZYY\dots Y$, while in the case of Lemma~\ref{lem:redpi} $\sigma$ is some inconsistent descriptive sequence. In either case, $F$ admits this descriptive sequence.

\textbf{Step 6:} Take a function $f_1$ that assigns one remaining part of $\fP^{(1)}$ to each vertex of $F$, in an order that is consistent with the descriptive sequence $\sigma$. For every twin pair of edges $e,e'$, consider the image under $f_1$ of their $2r-2$ vertices. Consider edges $g_e,g_{e'}$ of $R$ which certify that the corresponding $2r-2$-tuple of indices admits the descriptive sequence $\sigma$. Let $h_e,h_e'$ be edges in $H'$ corresponding to $g_e,g_{e'}$. Use these edges to construct the functions $f_2, f_3, f_4, \dots, f_r$ from the statement of the counting lemma, and apply the lemma to conclude the existence of a copy of $F$ in $H$.

\subsection{Proof of Theorem~\ref{thm:dens14}}

Fix a quasi-linear, head-tail-mixing $r$-graph $F$, which admits the descriptive sequence $XX\dots XZZYY\dots Y$. We will start by showing that $\pi_u(F)\geq 1/4$. To do this we will show that there exists a sequence of hypergraphs $H_n$ which is locally $1/4$-dense and which does not contain $F$ as a subgraph.

Take $[n]$ as the vertex set of $H_n$. Start by randomly coloring each pair of vertices in the colors red and blue. Then, for each $r$-tuple of vertices $v_1<v_2<\dots<v_r$, we place an edge of $H_n$ there if the pair $v_1v_2$ (that is, the head) is red and the pair $v_{r-1}v_r$ (the tail) is blue. It is clear that there is no pair of vertices that is both the head of some edge of $H_n$ and the tail of another, and as such $H_n$ cannot contain a head-tail-mixing hypergraph like $F$ as its subgraph.

To complete the upper bound, fix $\delta>0$. We will show that, with high probability, the hypergraph $H_n$ is $(1/4,\delta,\points)$-dense, by showing that the probability that a subset $S$ of $\delta n$ vertices has density less than $1/4-\delta$ is $o(2^{-n})$. As the pair coloring is random, the probability that each individual $r$-tuple becomes an edge of $H_n$ is $1/4$. Moreover, given a family of $r$-tuples in which no pair of them shares two vertices, the events in which each $r$-tuple becomes an edge of $H_n$ are independent.

By a theorem of R\"odl~\cite{Rod85}, for $n$ large enough it is possible to find $(\delta n)^2/(2r^2)$ $r$-tuples of vertices in $S$ which pairwise intersect in at most one vertex. Let $L$ be such a family, sampled uniformly at random. Let $E$ be the event in which less than a $1/4-\delta/2$ proportion of the $r$-tuples of $L$ become edges of $H_n$. On the one hand, if we first fix $L$ and then run the random coloring of pairs of vertices, by Chernoff's bound the probability of $E$ is at most $2^{-O(n^2)}$. On the other hand, suppose that the coloring procedure is run first. Let $P$ be the event in which $S$ has density lower than $1/4-\delta$. Conditioned on $P$, the expected proportion of $r$-tuples of $L$ which become edges of $H_n$ is at most $1/4-\delta$, since by symmetry each $r$-tuple has the same probability of being included in $S$. By Markov's inequality, the probability of $E$ conditioned on $P$ is at least $1-\frac{1/4-\delta}{1/4-\delta/2}\geq\delta$. Thus \[\Pr(E)\geq\Pr(P)\Pr(E|P)\quad\Longrightarrow\quad\Pr(P)\leq\frac{\Pr(E)}{\Pr(E|P)}=2^{-O(n^2)}.\] This concludes the proof of the upper bound $\pu(F)\geq 1/4$.

Next we will show that $\pu(F)\leq 1/4$. Let $\ell=|V(F)|$. Fix $\varepsilon>0$, and let $\{H_n\}_{n=1}^\infty$ be a locally $(1/4+\varepsilon)$-dense sequence of $r$-graphs. Our goal is to show that $F$ is a subgraph of $H_n$ for some $n$.

There will be several parameters involved in this proof. Describing them as a hierarchy is complicated by the fact that several inputs of the hypergraph regularity lemma are functions rather than constants. Among the ones that do behave like constants, the hierarchy is as follows:

\[r,\varepsilon^{-1},\ell\ll q_2,\mu\ll q_1\ll\xi_r\ll t, n_0, N, \delta^{-1}. \]

The specific relation between them, and with the functions $\xi_2, \xi_3, \dots, \allowbreak\xi_{r-1}, \beta$ will be described next. Our inputs consist of a value $\varepsilon>0$ and an $r$-graph $F$. Let $\ell=|V(F)|$. Let $q_2$ be the value of $k$ obtained from Lemma~\ref{lem:red14}, with inputs $\varepsilon/4$, $r$ and $\ell$. Let $\mu=\lceil4r^r/\varepsilon\rceil$, and $q_1$ be the Ramsey number $R_r(\max\{q_2,\mu\},2)$. Let $\xi_r=\frac{\varepsilon}{100q_1^r}$. We set $d_r=\varepsilon/100$, and then on Theorem~\ref{thm:count} with inputs $r,\ell,d_r$ we obtain a function $\xi_{r-1}=\xi_{r-1}(a_{r-1})$, then a function $\xi_{r-2}=\xi_{r-2}(a_{r-1}, a_{r-2})$, and so on, until we obtain functions $\xi_3=\xi_3(a_{r-1}, a_{r-2}, \dots, a_3)$, and then $\xi'_2,\beta,m_0=\xi_2',\beta,m_0(a_{r-1}, a_{r-2}, \dots, a_2)$, where $\xi'_2$ is the value $\xi_2$ from Theorem~\ref{thm:count} (we rename it because it will not match the one used as an input in Theorem~\ref{thm:regularity}). Let $\xi''_2(a_{r-1}, a_{r-2}, \dots, a_2)$ be the value of $\delta$ obtained from Lemma~\ref{lem:graphcount} on parameters $r$ and $\varepsilon/(4a_2^{\binom{r}{2}})$. Let $\xi_2=\min\{\xi_2',\xi_2''\}$. Next we apply Theorem~\ref{thm:regularity} with parameters $r,1/q_1,\xi_r,\allowbreak\beta,\allowbreak\xi_2,\xi_3,\dots,\xi_{r-1}$ to obtain $t$ and $n_0$. Finally we take $M=\max\left\{m_0(\vec a):a\in[t]^{r-2}\right\}$, $\delta=\min\{1/(2t),\varepsilon/100\}$ and $N=\max\{2n_0, t!, 2tM\}$. Choose $\tau$ such that $H_\tau$  has at least $N$ vertices and is $(1/4+\varepsilon,\delta,\points)$-dense.

Let $n'$ be the number of vertices of $H_\tau$, and let $n$ be the largest multiple of $t!$ less than or equal to $n'$. Since $n'\geq N\geq t!$, we have $n\geq n'/2$. Select a subset of $n$ vertices from $H_\tau$, and let $H$ be the hypergraph induced by $H_\tau$ on those vertices. Since $H_\tau$ is $(1/4+\varepsilon, \delta, \points)$-dense, $H$ is $(1/4+\varepsilon, 2\delta, \points)$-dense.

Since $n$ is at least $n_0$ and is divisible by $t!$, we can use the regularity lemma on $H$. Applying Theorem~\ref{thm:regularity} with parameters $r,1/q_1,\xi_r,\allowbreak\beta,\xi_2,\xi_3,\dots,\xi_{r-1}$, we obtain an $(r-1, \vec a)$-nested family of partitions $\fP$, for some vector $\vec a$, such that 

\begin{itemize}
    \item $\fP$ is $(1/q_1, \vec\xi(\vec a), \vec a, \beta(\vec a))$-equitable and $t$-bounded, and
    \item $H$ is $(*, \xi_r, \beta(\vec a))$-regular w.r.t. $\fP$.
\end{itemize}

Remove from $H$ all edges $e$ such that $H$ is not $(*,\xi_r,\beta(\vec a))$-regular over its polyad $\hat\cP^{(r-1)}(e)$, or if the density of $H$ over $\hat\cP^{(r-1)}(e)$ is less than $\varepsilon/100$, to obtain the hypergraph $H'$. Let $V_1, V_2, \dots, V_{a_1}$ be the classes of the partition $\fP^{(1)}$, and let $z=n/a_1$ be the size of each set $V_i$. 

We claim that the number of $r$-tuples of parts $V_{i_1}, V_{i_2}, \dots, V_{i_r}$ such that $H\setminus H'$ has edge density at least $\varepsilon/50$ on the $r$-partite graph between them is less than $\binom{a_1}{r}/\binom{q_1}{r}$. Indeed, within each $r$-tuple of parts the number of edges removed because they belong to sparse polyads is at most $\varepsilon z^r/100$, so to reach density $\varepsilon/50$ in $H\setminus H'$ we need to have at least $\varepsilon z^r/100$ edges in irregular polyads. Since $H$ is $(*,\xi_r, \beta(\vec a))$-regular over $\cP$, there are at most $\xi_rn^r= \xi_ra_1^rz^r$ edges of $H$ belonging to irregular polyads. The number of $r$-tuples of parts containing at least $\varepsilon z^r/100$ of these edges is at most $\xi_ra_1^rz^r/(\varepsilon z^r/100)= (a_1/q_1)^r<\binom{a_1}{r}/\binom{q_1}{r}$.

Select a set $I_1\subseteq [a_1]$ of size $q_1$ uniformly at random (we know that $a_1\geq q_1$ because $\fP$ is $(1/q_1, \vec\xi(\vec a), \vec a, \beta(\vec a))$-equitable, where the first parameter means that $|\cross_r(\fP^{(1)})|\geq \binom{n}{r}/q_1$, which is impossible if $\fP^{(1)}$ has fewer than $q_1$ parts). The expected number of $r$-tuples $i_1, i_2, \dots, i_r$ in $I_1$ such that $H\setminus H'$ has edge-density at least $\varepsilon/100$ on the corresponding parts $V_i$ is less than 1, therefore there exists a set $I_1$ of size $q_1$ in which the density of $H\setminus H'$ in each corresponding $r$-tuple of parts is less than $\varepsilon/100$. W.l.o.g., we can assume that $I_1=[q_1]$.

Consider an auxiliary complete $r$-graph on the vertex set $[q_1]$. We color an $r$-tuple of elements in red if the edge density of $H$ in the corresponding $r$-tuple of parts is at least $1/4+2\varepsilon/3$, and blue otherwise. Since $q_1=R_r(\max\{q_2, \mu\},2)$, there is either a red clique of size $q_2$ or a blue clique of size $\mu$. We claim that the blue clique is impossible.  Indeed, consider the set $W$ formed by the union of the vertex sets $V_i$ with $i$ in this blue clique. Its size is $|W|=\mu z\geq z=n/a_1\geq n/t\geq 2\delta n$, so from the fact that $H$ is $(1/4+\varepsilon,2\delta, \points)$-dense, and $\delta\geq \varepsilon/100$, the edge-density of $H$ on the set $W$ is at least $1/4+49\varepsilon/50$. On the other hand, every $r$-tuple of parts in $W$ has edge-density at most $1/4+2\varepsilon/3$, and the number of edges in $W$ with two vertices in the same part is at most $\mu^{r-1}z^r$. But that means that the number of edges of $H$ in $W$ is at most $(1/4+2\varepsilon/3)\binom{\mu z }{r}+\mu^{r-1}z^r\leq (1/4+2\varepsilon/3)\binom{\mu z}{r}+r^r\binom{\mu z}{r}/\mu<(1/4+49\varepsilon/50)\binom{\mu z}{r}$, reaching a contradiction. There exists a red clique of size $q_2$, and let $I_2\subseteq I_1$ be the set of elements in the clique. Within each $r$-tuple of parts in $I_2$ the density of $H$ is at least $1/4+2\varepsilon/3$, and the density of $H'\setminus H$ is at most $\varepsilon/50$, so the density of $H'$ on each $r$-tuple is at least $1/4+\varepsilon/2$. W.l.o.g., we can assume that $I_2=[q_2]$.

Consider the partition $\fP^{(2)}$, which partitions $\cross_2(\fP^{(1)})\subseteq \binom{V(H)}{2}$. Because $\fP$ is $(r-1,\vec a)$-nested, for every pair of parts $(V_i,V_j)$, $\fP^{(2)}$ partitions the set of pairs of vertices between $V_i$ and $V_j$ into $a_2$ parts. We can think of each such part $\cP=\cP^{(2)}$ as a bipartite graph between $V_i$ and $V_j$. Moreover, since $\fP$ is $(1/q_1, \vec\xi(\vec a), \vec a, \beta(\vec a))$-equitable, each bipartite graph $\cP$ is $(1/a_2,\xi_2(\vec a))$-regular, in the sense of Lemma~\ref{lem:graphcount}.

We will define an auxiliary $(q_2,r)$-reduced graph $R$. Remember that a $(k,r)$-reduced graph is ar $\binom{r}{2}$-graph whose vertex set is partitioned into $\binom{k}{2}$ parts $W_{i,j}$, with $1\leq i<j\leq k$. Each edge has $\binom{r}{2}$ vertices, which lie in parts of the form $W_{i_1, i_2}, W_{i_1, i_3}, \dots, W_{i_r-1,i_r}$, for some $1\leq i_1<i_2<\dots<i_r\leq k$. The constituent $\mathcal{A}_{i_1, i_2, \dots, i_r}$ is the $\binom{r}{2}$-partite $\binom{r}{2}$-graph induced by $R$ on the vertex sets $W_{i_j,i_{j'}}$.

Let the vertex sets of $R$ be $W_{i,j}$ with $1\leq i<j\leq q_2$. Each $W_{i,j}$ will have size $a_2$, and the elements of $W_{i,j}$ will represent the $a_2$ partition classes $\cP\in\fP^{(2)}$ in which the bipartite graph between $V_i$ and $V_j$ gets split by $\fP^{(2)}$. Let $i_1, i_2,\dots, i_r$ be an $r$-tuple of elements of $[q_2]$, and for each $1\leq j<j'\leq r$ let $w_{j,j'}$ be a vertex in $W_{i_j,i_{j'}}$, and let $\cP_{j,j'}$ be the corresponding part of $\fP^{(2)}$. We define the edges of $R$ as follows: the $\binom{r}{2}$ vertices $w_{1,2}, w_{1,3}, \dots, w_{r-1,r}$ form an edge of $R$ if there exists an edge $e\in H'$ in which each pair of vertices belongs to one of the classes $\cP_{1,2}, \cP_{1,3}, \dots, \cP_{r-1,r}$.

Consider the constituent $\mathcal{A}_{1,2,\dots, r}$. Every edge of $e\in H'$ with one vertex in each of $V_1,V_2,\dots,V_r$ generates in $R$ an edge, which we will denote by $R(e)$, whose $\binom{r}{2}$ vertices correspond to the parts of $\fP^{(2)}$ containing each pair of vertices in $e$. There are at least $(1/4+\varepsilon/2)z^r$ edges in $H'$ between $V_1,V_2,\dots,V_r$. To bound the number of edges in $\mathcal{A}_{1,2,\dots,r}$, we will bound how many edges in $H'$ can generate the same edge of $R$.

Let $w_{1,2}, w_{1,3},\dots, w_{r-1,r}$ be vertices, with $w_{i,j}\in W_{i,j}$, and let $\cP_{i,j}$ be the part of $\fP^{(2)}$ corresponding to $w_{i,j}$. Treating each $\cP_{i,j}$ as a bipartite graph, each edge $e\in H'$ with $R(e)=w_{1,2}w_{1,3}\dots w_{r-1,r}$ corresponds to an $r$-clique in the $r$-partite graph formed by the union of the graphs $\cP_{i,j}$. Because each $\cP_{i,j}$ is $(1/a_2,\xi_2(\vec a))$, and $\xi_2(\vec a)\leq \xi''_2(\vec a)$, we can apply the graph counting lemma. By Lemma~\ref{lem:graphcount} with parameters $r$ and $\varepsilon/(4a_2^{\binom{r}{2}})$, the number of $r$-cliques in the union of the graphs $\cP_{i,j}$ is at most $z^r/a_2^{\binom{r}{2}}+\varepsilon z^r/(4a_2^{\binom{r}{2}})=(1+\varepsilon/4)z^r/a_2^{\binom{r}{2}}$.

Because $H$ has at least $(1/4+\varepsilon/2)z^r$ edges between the sets $V_1, V_2, \dots, V_r$, and each edge of the constituent $\mathcal{A}_{1,2,\dots, r}$ corresponds to at most $(1+\varepsilon/4)z^r/a_2^{\binom{r}{2}}$ of those edges, the number of edges in $\mathcal{A}_{1,2,\dots,r}$ is at least $(1/4+\varepsilon/2)a_2^{\binom{r}{2}}/(1+\varepsilon/4)\geq (1/4+\varepsilon/4)a_2^{\binom{r}{2}}$. In other words, the constituent has edge density at least $1/4+\varepsilon/4$. The same holds true for all other constituents, so the $(q_2,r)$-reduced graph $R$ is $(1/4+\varepsilon/4)$-dense.

By the choice of $q_2$, by Lemma~\ref{lem:red14} there exists a set $I_3$ of $\ell$ indices in $I_2=[q_2]$ such that every $2r-2$-tuple of indices in $I_3$ admits the descriptive sequence $XX\dots XZZYY\dots Y$. W.l.o.g., we can assume that $I_3=[\ell]$. The hypergraph $F$ admits $XX\dots XZZYY\dots Y$, so there exists an order $v_1, v_2,\dots, v_\ell$ of its vertex set in which $XX\dots XZZYY\dots Y$ describes every pair of edges with intersection size two. To conclude the proof of Theorem~\ref{thm:dens14}, our goal is to use the hypergraph counting lemma (Theorem~\ref{thm:count}) to prove that $H$ contains $F$ as a subgraph.

$F$ is quasi-linear, meaning that non-twin pairs of edges intersect in at most one vertex. Any twin pair of edges $e,e'$ has the form $e=v_{i_1}, v_{i_2}, \dots, v_{i_r}$ and $e'=v_{i_{r-1}}, v_{i_r}, \dots, v_{i_{2r-2}}$, with $i_1\leq i_2\leq\dots\leq i_{2r-2}$. The set of $2r-2$ indices $\{i_1, \dots, i_{2r-2}\}$ admits the descriptive sequence $XX\dots XZZYY\dots Y$, so in $R$ there exist edges $g_e\in\mathcal{A}_{i_1, i_2, \dots, i_r}$ and $g_{e'}\in\mathcal{A}_{i_{r-1}, i_r, \dots, i_{2r-2}}$ which intersect in one vertex $w_{i_{r-1},i_r}\in W_{i_{r-1}, i_r}$. There exist edges $h_e,h_{e'}\in H'$ such that $R(h_e)=g_e$ and $R(h_{e'})=g_{e'}$. Each vertex $v_{i_j}$ in $e$ corresponds to a vertex in $h_e$ or $h_e$, namely the one contained in $V_{i_j}$. Note that $h_e$ and $h_{e'}$ each have one vertex in each of $V_{i_{r-1}}$ and $V_{i_r}$, and both pairs of vertices belong to the same class of $\fP^{(2)}$, the one corresponding to $w_{i_{r-1}, i_r}$.

We are now ready to define the functions $f_i:F_{<i>}\rightarrow\fP^{(i)}$ with $1\leq i\leq r-1$ from Theorem~\ref{thm:count}. For any $i$-tuple of vertices $S$ in $F_{<i>}$, choose an edge $e$ containing these vertices. Then $f_i(S)$ is the part of $\fP^{(i)}$ containing the corresponding $i$ vertices in $h_e$. We define $f_r(e)=\hat\fP^{(r-1)}(h_e)$ for each $e\in E(F)$.

We begin by observing that the functions $f_i$ are all well-defined. The image of each vertex $v_i$ is the part $V_i$. If a pair of vertices $v_i,v_j$ is contained in two edges $e,e'$, then $e$ and $e'$ must form a twin pair, in which case the part of $\fP^{(2)}$ containing the corresponding pair of vertices in $h_e$ and $h_e'$ is the same. There is no set of three or more vertices contained in more than one edge. In addition, if $2\leq i\leq r-1$, and $S$ is a set of $i$ vertices, we have that the $(i-1)$-th polyad of $f_i(S)$ is precisely $\cup_{v\in S}f_{i-1}(S-v)$, and we also have $f_r(e)=\hat \fP^{(r-1)}(h_e)=\cup_{v\in e}\cP^{(r-1)}(e-v)=\cup_{v\in e}f_{r-1}(e-v)$ for all $e\in E(F)$. Finally, because each $h_e$ belongs to $H'$, and in particular it was not removed from $H$ when creating $H'$, it means that $H$ is $(d,\xi_r,\beta(\vec a))$-regular w.r.t. $f_r(e)$ for some $d\geq \varepsilon/100$. Observe also that the size of each vertex set $V_i$ is at least $m_0(\vec a)$, because $z=n/a_1\geq n'/2a_1\geq N/2t\geq M\geq m_0(\vec a)$. We can apply Theorem~\ref{thm:count} to deduce that $H$ contains $F$ as a subgraph, and so does $H_\tau$, as we wanted to show. That completes the proof of Theorem~\ref{thm:dens14}.

\subsection{Proof of Theorem~\ref{thm:denspi}}

Let $F$ be a quasi-linear, inconsistent $r$-graph which admits all inconsistent descriptive sequences of order $r$. We will start by showing that $\pu(F)\geq\pi_r$, by constructing a locally $\pi_r$-dense sequence of $r$-graphs $\{H_n\}_{n=1}^\infty$ which are $F$-free. We will take $[n]$ as the vertex set of $H_n$. Let $\varphi_n:\binom{[n]}{2}\rightarrow \binom{[r]}{2}$ be a random function, i.e., a uniformly random coloring of the pairs of vertices of $H_n$ with the set of colors $\binom{[r]}{2}$. The edges of $H_n$ will be precisely the $r$-tuples of vertices $1\leq v_1< v_2<\dots<v_r\leq n$ satisfying that $\varphi_n(\{v_i,v_j\})=\{i,j\}$ for all $1\leq i<j\leq n$.

For each $\delta>0$ we will first show that, with high probability, $H_n$ is $(\pi_r, \delta, \points)$-dense. Let $S$ be a set of $\delta n$ vertices from $[n]$. We will show that $S$ has edge-density at least $\pi_r-\delta$ with probability $1-o(2^{-n})$. 

By a theorem of R\"odl~\cite{Rod85}, for $n$ large enough it is possible to find $(\delta n)^2/(2r^2)$ $r$-tuples of vertices in $S$ pairwise intersecting in at most one vertex. Let $L$ be such a family, sampled uniformly at random. For each individual edge $e$ in $L$, the probability that $\varphi_n$ induces on $e$ the unique coloring that makes $e$ an edge of $H_n$ is $\binom{r}{2}^{-\binom{r}{2}}=\pi_r$, and the events are mutually independent among the edges of $L$. Let $E$ be the event that fewer than $(\pi_r-\delta/2)|L|$ elements of $L$ become edges of $H_n$. By Chernoff's bound, the probability of $E$ is at most $2^{-O(n^2)}$. On the other hand, suppose that $\varphi_n$ is sampled first, and $L$ is sampled later. Let $P$ be the event that $S$ has edge-density lower than $\pi_r-\delta$. By symmetry, each $r$-tuple of vertices in $S$ is included in $L$ with equal probability, so by linearity of expectation, the expected proportion of elements of $L$ which are edges of $H_n$ equals the edge-density of $S$. By Markov's inequality, the probability of $E$ conditioned on $P$ is at least $1-\frac{\pi_r-\delta}{\pi_r-\delta/2}$. Thus
\[\Pr(E)\geq\Pr(P)\Pr(E|P)\quad\Longrightarrow\quad\Pr(P)\leq\frac{\Pr(E)}{\Pr(E|P)}=2^{-O(n^2)}=o(2^{-n}).\]

On the other hand, we can show that $F$ is not a subgraph of $H_n$ for any $n$. Indeed, suppose that there is a copy of $F$ in $H_n$, and consider the vertex-order $\preceq$ of $V(F)$ induced by the natural order of $V(H_n)=[n]$. Since $F$ is inconsistent, the order $\preceq$ is not a consistent order in $F$, and there exists a twin pair of edges $e, e'$ such that $e\cap e'$ plays different roles in $e$ and $e'$. On the other hand, let $u$ and $v$ be the vertices of $H_n$ corresponding to $e\cap e'$. By construction of $H_n$, if $\varphi_n(u,v)=\{i,j\}$, then the pair $\{u,v\}$ plays the role $\{i,j\}$ in all the edges containing both $u$ and $v$, yielding a contradiction. We conclude that the sequence $\{H_n\}_{n=1}^\infty$ is $F$-free. This proves that $\pu(F)\geq \pi_r$.

The proof of the upper bound $\pu(F)\leq\pi_r$ is analogous to that of Theorem~\ref{thm:dens14}. The main difference is that, once we construct the $(q_2,r)$-reduced graph $R$, we apply Lemma~\ref{lem:redpi} rather than Lemma~\ref{lem:red14}. We obtain a collection $I_3$ of $\ell$ indices and an inconsistent descriptive sequence $\sigma=(s_1, s_2, \dots, s_{2r-2})$ such that all $2r-2$-tuples from $I_3$ admit $\sigma$. Since $F$ admits all inconsistent descriptive sequences, in particular $\sigma$, there exists an order $v_1, v_2, \dots, v_\ell$ of $V(F)$ in which every twin pair of edges of $F$ is described by $\sigma$. From this point, we can finish as in the proof of Theorem~\ref{thm:dens14}.

\section{Concluding remarks}\label{sec:concluding}

In~\cite{ReiRS18}, Reiher, R\"{o}dl and Schacht characterized the $3$-graphs $F$ with $\pu(F)=0$. Their characterization can be described as follows: there exists an ordering $\preceq$ of $V(F)$, and a function $\varphi:\binom{V(F)}{2}\rightarrow\binom{[3]}{2}$ such that every pair of vertices $u,v$ plays the role $\varphi(uv)$ in every edge containing both $u$ and $v$.

A consequence of this characterization is that every $3$-graph $F$ has either $\pu(F)=0$ or $\pu(F)\geq 1/27$. To see this, consider a random function $\varphi_n:\binom{[n]}{2}\rightarrow \binom{[3]}{2}$, and let $H_n$ be the $3$-graph on vertex set $[n]$ where the edges are the triples $1\leq u<v<w\leq n$ with $\varphi_n(uv)=\{1,2\}$, $\varphi_n(uw)=\{1,3\}$ and $\varphi_n(vw)=\{2,3\}$ (this is the construction from the lower bound of Theorem~\ref{thm:denspi} in the particular case $r=3$). We showed that with high probability this sequence is locally $1/27$-dense. If $F$ does not appear as a subgraph of any $H_n$, then by definition $\pu(F)\geq 1/27$, whereas if $F$ appears as a subgraph, then $F$ satisfies the property from the characterization, so $\pu(F)=0$. Together with the construction of a $3$-graph $F$ with $\pu(F)=1/27$ found in~\cite{GarKL24}, we conclude that $1/27$ is the minimum positive value of the uniform Tur\'an density of $3$-graphs.

We conjecture that a similar characterization holds for $r$-graphs for general $r$.

\begin{conjecture}\label{conj:dens0}
    Let $F$ be an $r$-graph. Then $\pu(F)=0$ if and only if there exists an ordering $\preceq$ on $V(F)$ and a function $\varphi:\binom{V(F)}{2}\rightarrow\binom{[r]}{2}$ with the property that, for each pair of vertices $u,v$ and every edge $e$ containing $u$ and $v$, the pair $\{u,v\}$ plays the role $\varphi(uv)$ in $e$.
\end{conjecture}

If Conjecture~\ref{conj:dens0} holds, then every $r$-graph with $\pu(F)>0$ satisfies $\pu(F)\geq\pi_r$. Indeed, we can construct a locally $\pi_r$-dense sequence of $r$-graphs as in the proof of the upper bound of Theorem~\ref{thm:denspi}, and once again any $F$ which appears as a subgraph satisfies $\pu(F)=0$. Therefore, Theorem~\ref{thm:mainpi} implies that $\pi_r$ is the minimum positive value for the uniform Tur\'an density of $r$-graphs.

More ambitiously, one could aim to prove an analogue of the main result of~\cite{Lam24}: the uniform Tur\'an density of an $r$-graph $F$ is the supremum of the densities of the generalized palette constructions not containing $F$ as a subgraph. For simplicity, we will focus on the case in which the edges of $F$ pairwise intersect in at most two vertices.

Let $\mathcal{C}$ be a finite set, whose elements we will call \emph{colors}, and let $\mathcal{A}\subseteq\mathcal{C}^{\binom{r}{2}}$ be a set, whose elements we call \emph{admissible patterns}. We call the pair $\mathcal{P}=(\mathcal{C},\mathcal{A})$ a \emph{palette}, and its density is $d(\mathcal{P})=|\mathcal{A}|/|\mathcal{C}|^{\binom{r}{2}}$. We can use the palette $\mathcal{P}$ to generate a random $r$-graph $H_n$ on the vertex set $[n]$ as follows. Take a random $\mathcal{C}$-coloring of the graph $K_n$, that is, a uniformly random function $\varphi:\binom{[n]}{2}\rightarrow\mathcal{C}$. For every $r$-tuple of vertices $1\leq v_1<v_2<\dots<v_r\leq n$, these vertices form an edge of $H_n$ iff $(\varphi(v_1v_2),\varphi(v_1v_3),\dots, \varphi(v_{r-1}v_r))\in\mathcal{A}$.

Using standard concentration inequalities, one can show that with high probability the generated sequence $\{H_n\}_{n=1}^\infty$ is locally $d(\mathcal{P})$-dense. Therefore, if $F$ is not contained in any hypergraph in the sequence, we have $\pu(F)\geq d(\mathcal{P})$. Notice that the lower bound constructions from Theorem~\ref{thm:dens14} and Theorem~\ref{thm:denspi} are particular cases of palettes. In the former case we can take $\mathcal{C}=\{\texttt{red},\texttt{blue}\}$ and $\mathcal{A}$ contains the $2^{\binom{r}{2}-2}$ vectors whose first entry is \texttt{red} and whose last entry is \texttt{blue}. In the latter case we have $\mathcal{C}=\binom{[r]}{2}$ and $\mathcal{A}$ only contains the admissible pattern $\{\{1,2\},\{1,3\},\dots,\{r-1,r\}\}$.

One can hope to prove that for every $r$-graph $F$ where edges intersect pairwise in at most two vertices, $\pu(F)$ is the supremum of the densities of palettes such that the randomly generated sequence $H_n$ does not contain $F$ as a subgraph. However, for any $r\geq 4$, attempting to mirror the proof from~\cite{Lam24} fails at a crucial step.

Let $\mathcal{P}=(\mathcal{C},\mathcal{A})$ be a palette, and let $k\geq r$ be a positive integer. We define the $k$-blowup of $\mathcal{P}$, denoted by $\mathcal{P}[k]$, as the following $(k,r)$-reduced hypergraph. Each vertex set $V_{i,j}$ with $1\leq i<j\leq k$ is a copy of $\mathcal{C}$. We denote the copy of the color $c\in\mathcal{C}$ in $V_{i,j}$ as $c^{i,j}$. For each $1\leq t_1<t_2<\dots<t_r\leq k$, the constituent $\mathcal{A}_{t_1,t_2,\dots,t_r}$ has exactly $|\mathcal{A}|$ edges: for each $(c_{1,2},c_{1,3},\dots,c_{r-1,r})\in\mathcal{A}$, the constituent contains the edge $(c_{1,2})^{t_1,t_2}(c_{1,3})^{t_1t_3}\cdots (c_{r-1,r})^{t_{r-1},t_r}$. The $(k,r)$-reduced graph $\mathcal{P}[k]$ is $d(\mathcal{P})$-dense.

Given a $(k,r)$-reduced graph $H$ on vertex sets $V_{i,j}$ and a $(k',r)$-reduced graph $H'$ on vertex sets $W_{i,j}$, a homomorphism from $H$ to $H'$ is a pair $(\varphi,f)$, where $\varphi$ is a function $\varphi:V(H)\rightarrow V(H')$ that maps edges to edges, and $f$ is a function $f:[k]\rightarrow [k']$ such that if $v\in V_{i,j}$ then $\varphi(v)\in W_{f(i),f(j)}$. The following (implicit) lemma is the basis of the main result from~\cite{Lam24}:

\begin{lemma}
    For every $\varepsilon>0$ and every $m$ there exists $k$ with the following property: every $d$-dense $(k,3)$-reduced graph admits a homomorphism from $\mathcal{P}[m]$, for some palette $\mathcal{P}$ with $d(\mathcal{P})\geq d-\epsilon$.
\end{lemma}

The same lemma is false for $(k,4)$-reduced graphs: for every $k$, there exists a $1/4$-dense $(k,4)$-reduced graph $H$ which does not contain the $17$-blowup of any palette with positive density. Let each $V_{i,j}$ be a copy of $2^{[k]}$. For every $S\subseteq[k]$, we denote the copy of $S$ in $V_{i,j}$ as $S^{i,j}$. For every $1\leq t_1<t_2<t_3<t_4\leq k$, \[((S_1)^{t_1,t_2},(S_2)^{t_1,t_3}, (S_3)^{t_1,t_4},(S_4)^{t_2,t_3},(S_5)^{t_2,t_4},(S_6)^{t_3,t_4})\in H\quad\text{iff}\quad t_3\in S_1, t_4\notin S_1.\]

$H$ is $1/4$-dense because for any $t_1<t_2<t_3<t_4$ exactly one quarter of the sets in $2^{[k]}$ contain $t_3$ but not $t_4$. On the other hand, it is possible to show that $H$ does not contain $\mathcal{P}[17]$ for any palette $\mathcal{P}$ such that $\mathcal{A}$ is non-empty, using the fact that by the Erd\H{o}s-Szekeres theorem, the function $f$ in any homomorphism is monotone on a sequence of five indices. Examples like this are a natural obstacle to characterizing uniform Tur\'an density in terms of palettes for $r$-graphs with $r\geq 4$.

More generally, if the edges of $F$ are allowed to have arbitrary intersections, palettes should be modified to also color triples of vertices, quadruples of vertices, and so on up until $(r-1)$-tuples of vertices. In order to decide whether an $r$-tuple of vertices in $n$ is an edge of $H_n$ or not, we should consider the color pattern of its pairs of vertices, its triples, and so on up until $(r-1)$-tuples of vertices. In this case, $\mathcal{A}\subseteq\mathcal{C}^{2^r-r-2}$, where $2^r-r-2=\sum_{i=2}^{r-1}\binom{r}{i}$. If it holds that the uniform Tur\'an density of an $r$-graph $F$ is the supremum of the densities of palette constructions that do not contain $F$ as a subgraph, then Conjecture~\ref{conj:dens0} would follow.

\section*{Acknowledgments}

The author would like to thank Frederik Garbe, Daniel Il'kovi\v{c}, Dan Kr\'al' and Filip Ku\v{c}er\'ak for helpful discussions on the topic, and Hong Liu for his comments on an early version of this paper.

\bibliographystyle{abbrv}
\bibliography{bibliog.bib}

\end{document}